\documentclass[a4paper,12pt]{amsart}

\usepackage{amsmath, amsthm, amscd, amsfonts, amssymb, graphicx, color}

\newcommand*{\medcap}{\mathbin{\scalebox{1.5}{\ensuremath{\cap}}}}%

\usepackage{enumitem}
\usepackage{url}
\usepackage{multirow}
\usepackage{subfigure}
\usepackage[pass]{geometry}
\usepackage{cite}
\usepackage{cancel}
\usepackage{color}

\usepackage{lscape}

\usepackage{tikz}
\usetikzlibrary{calc,shapes,decorations.pathreplacing,matrix}
\usetikzlibrary{patterns}

\tikzset{square matrix/.style={
    matrix of nodes,
    column sep=-\pgflinewidth, row sep=-\pgflinewidth,
    nodes={draw,
      minimum height=4.5pt,
      anchor=center,
      text width=4.5pt,
      align=center,
      inner sep=0pt
    },
  },
  square matrix/.default=1.2cm
}

\newtheorem{theo}{Theorem}[section]
\newtheorem{lemma}[theo]{Lemma}

\newtheorem{corol}[theo]{Corollary}
\newtheorem{conj}[theo]{Conjecture}
\newtheorem{prop}[theo]{Proposition}
\newtheorem{remark}[theo]{Remark}

\begin{document}

\title{$3$-tuple total domination number of rook's graphs}

\author{Behnaz Pahlavsay}
\address{Behnaz Pahlavsay, Department of Mathematics, Hokkaido University, Kita 10, Nishi 8, Kita-Ku, Sapporo 060-0810, Japan.}
\email{pahlavsay@math.sci.hokudai.ac.jp}
\author{Elisa Palezzato}
\address{Elisa Palezzato, Department of Mathematics, Hokkaido University, Kita 10, Nishi 8, Kita-Ku, Sapporo 060-0810, Japan.}
\email{palezzato@math.sci.hokudai.ac.jp}
\author{Michele Torielli}
\address{Michele Torielli, Department of Mathematics, GI-CoRE GSB, Hokkaido University, Kita 10, Nishi 8, Kita-Ku, Sapporo 060-0810, Japan.}
\email{torielli@math.sci.hokudai.ac.jp}

\date{\today}

\begin{abstract}
A $k$-tuple total dominating set ($k$TDS) of a graph $G$ is a set $S$ of vertices in which every vertex in $G$ is adjacent to at least $k$ vertices in $S$. The minimum size of a $k$TDS is called the
$k$-tuple total dominating number and it is  denoted by $\gamma_{\times k,t}(G)$.  We give a constructive proof of a general formula for $\gamma_{\times 3, t}(K_n \Box K_m)$.
%
%
\end{abstract}
\maketitle

\section{Introduction}

Domination is well-studied in graph theory and the literature on this subject has been surveyed and detailed in the two books by Haynes, Hedetniemi, and Slater~\cite{HHS5, HHS6}. Among the many variations of domination, the one relevant to this paper is $k$-tuple total domination, which was introduced by Henning and Kazemi \cite{HK8} as a generalization of \cite{HH3}. Throughout this paper, we use standard notation for graphs, see for example \cite{bondy2008graph}.  All graphs considered here are finite, undirected, and simple.

%
For a graph $G=(V_G,E_G)$ and $k \geq 1$, a set $S \subseteq V_G$ is called a $k$-\emph{tuple total dominating set} ($k$TDS) if every vertex $v \in V$ has at least $k$ neighbours in $S$, i.e., $|N_G(v)\cap S| \geq k$.  The \emph{$k$-tuple total domination number}, which we denote by $\gamma_{\times k,t}(G)$, is the minimum cardinality of a $k$TDS of $G$.  
We use min-$k$TDS to refer to  $k$TDSs of minimum size.

An immediate necessary condition for a graph to have a $k$-tuple total dominating set is that every vertex must have at least $k$ neighbours.  For example, for $k \geq 1$, a $k$-regular graph $G=(V_G,E_G)$ has only one $k$-tuple total dominating set, namely $V_G$ itself.

In the history of domination problems, a lot of work has been done to study the class of cartesian product of graphs and in particular of rook's graphs. 
Given two graphs $G$ and $H$, their \emph{Cartesian product} $G \Box H$ is the graph with vertex set $V_G\times V_H$ where two vertices $(u_{1},v_{1})$ and $(u_{2},v_{2})$ are adjacent if and only if either $u_{1}=u_{2}$ and $v_{1}v_{2}\in E_H$ or $v_{1}=v_{2}$ and $u_{1}u_{2}\in E_G$.  For more information on the cartesian product of graphs see \cite{IK}.  We will be particularly interested in the case when $K_n \Box K_m$, where $K_n$ is the complete graph on $n$ vertices. Such graph is known as the $n \times m$ \emph{rook's graph}, as edges represent possible moves by a rook on an $n \times m$ chess board.  The $3 \times 4$ rook's graph is drawn in Figure~\ref{fi:rook34}, along with a min-$3$TDS.

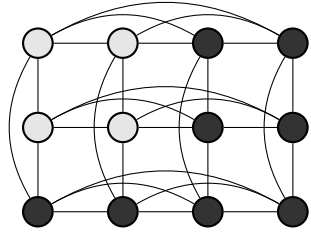
\begin{figure}[htp]
\centering
\begin{tikzpicture}
\matrix[nodes={draw, thick, fill=black!10, circle},row sep=0.7cm,column sep=0.7cm] {
  \node(11){}; &
  \node(12){}; &
  \node[fill=black!80](13){}; &
  \node[fill=black!80](14){}; \\
  \node(21){}; &
  \node(22){}; &
  \node[fill=black!80](23){}; &
  \node[fill=black!80](24){}; \\
  \node[fill=black!80](31){}; &
  \node[fill=black!80](32){}; &
  \node[fill=black!80](33){}; &
  \node[fill=black!80](34){}; \\
};
\draw (11) to (12) to (13) to (14); \draw (11) to[bend left] (13); \draw (12) to[bend left] (14); \draw (11) to[bend left] (14);
\draw (21) to (22) to (23) to (24); \draw (21) to[bend left] (23); \draw (22) to[bend left] (24); \draw (21) to[bend left] (24);
\draw (31) to (32) to (33) to (34); \draw (31) to[bend left] (33); \draw (32) to[bend left] (34); \draw (31) to[bend left] (34);
\draw (11) to (21) to (31) to[bend left] (11);  \draw (12) to (22) to (32) to[bend left] (12);  \draw (13) to (23) to (33) to[bend left] (13);  \draw (14) to (24) to (34) to[bend left] (14);
\end{tikzpicture}
\caption{The $3 \times 4$ rook's graph, i.e., $K_3 \Box K_4$.  The dark vertices form a min-$3$TDS, so $\gamma_{\times 3,t}(K_3 \Box K_4)=8$.}\label{fi:rook34}
\end{figure}

In \cite{Viz}, Vizing studied the \emph{domination number} of graphs, i.e. the minimal cardinality of a dominating set, and made an elegant conjecture that has subsequently become one the most famous open problems in domination theory.
\begin{conj}[Vizing's Conjecture]
For any graphs $G$ and $H$, $$\gamma(G)\gamma(H) \leq \gamma(G\Box H),$$
where $\gamma(G)$ and $\gamma(H)$ are the domination numbers of the graphs $G$ and $H$, respectively.
\end{conj}
Over more than forty years (see \cite{BDG} and references therein), Vizing's Conjecture has been shown to hold for certain restricted classes of graphs, and furthermore, upper and lower bounds on the inequality have gradually tightened.  Additionally, researcher have explored inequalities (including Vizing-like inequalities) for different variations of domination \cite{HHS6}. A significant breakthrough occurred when in \cite{CS1} Clark and Suen proved that  $$\gamma (G)\gamma (H)\leq 2\gamma (G\Box H)$$ which led to the discovery of a Vizing-like inequality for total domination \cite{HR9,HPT}, i.e.,
\begin{equation}\label{eq:HRbound}
\gamma_t(G)\gamma_t(H) \leq 2 \gamma_t(G \Box H),
\end{equation}
as well as for paired \cite{BHRo,HJ,CM}, and fractional domination \cite{FRDM}, and the $\{k\}$-domination function (integer domination) \cite{BM,HL,CMH}, and total $\{k\}$-domination function \cite{HL}.

Burchett, Lane, and Lachniet \cite{BLL} and Burchett \cite{B2011} found bounds and exact formulas for the $k$-tuple domination number and $k$-domination number of the rook's graph in square cases, i.e., $K_n \Box K_n$ (where $k$-domination is similar to {$k$-tuple total domination}, but only vertices outside of the domination set need to be dominated).  The $k$-tuple total domination number is known for $K_n \times K_m$ \cite{HK7} and bounds are given for supergeneralized Petersen graphs \cite{KP}.  In \cite{KPS}, the authors showed that the graph $K_n \Box K_m$ is an extremal case in the study of $k$TDS of cartesian product of graphs, motivating the study of the class of rook's graphs. Specifically, they showed that  $$\gamma_{\times k,t}(K_n \Box K_m) \leq \gamma_{\times k,t}(G \Box H),$$ when $G$ and $H$ are two graphs with $n$ and $m$ vertices, respectively. Moreover, they computed $\gamma_{\times 2, t}(K_n \Box K_m)$ for all $m\ge n$.

This paper is organized as follows. In Section 2, we recall basic properties on $k$TDS. In Section 3, we describe a special class of $3$TDS matrices. In Section 4, we describe several useful inequalities for $\gamma_{\times 3, t}(K_n \Box K_m)$.  In Section 5, we compute $\gamma_{\times 3, t}(K_n \Box K_n)$, for any $n\ge3$. In Section 6, we describe our main result: we determine the value of $\gamma_{\times 3, t}(K_n \Box K_m)$ in Theorem~\ref{theo:thegeneralcase} for all $m\ge n$.

\section{Preliminares}
We recall some basic properties of $k$TDS and their relations with $(0,1)$-matrices.
Assume the vertex set of the complete graph $K_n$ is $[n]:=\{1,\dots, n\}$. Given $D\subseteq V_{K_n}\times V_{K_m}$, we can associate to it a $n \times m$ $(0,1)$-matrix $S=(s_{ij})$ with $s_{ij}=1$ if and only if $(i,j) \in D$.
Let $S=(s_{ij})$ be a $n \times m$ $(0,1)$-matrix. Define 
\begin{eqnarray*}
\kappa_S(i,j) & =&\overbrace{\left( \textstyle\sum_{r \in [m]} s_{ir} \right)}^{\text{$i$-th row sum}} + \overbrace{\left( \textstyle\sum_{r \in [n]} s_{rj} \right)}^{\text{$j$-th column sum}} - 2s_{ij}\\
 & =&\mathfrak{r}_S(i) + \mathfrak{c}_S(j) - 2s_{ij}.
 \end{eqnarray*}
If no confusion arises, we will simply write $\mathfrak{r}(i)$, $\mathfrak{c}(j)$ and $\kappa(i,j)$. Notice that $\mathfrak{r}(i)$ is the number of ones in the $i$-th row of $S$ and, 
similarly, $\mathfrak{c}(j)$ is the number of ones in the $j$-th column of $S$.
Moreover, we will denote by $|S|$ the number of ones in $S$.

A  $n \times m$ $(0,1)$-matrix $S=(s_{ij})$ corresponds to a  $k$TDS $D$ of $K_n \Box K_m$ if and only it satisfies
\begin{equation*}
\kappa(i,j) \geq k
\end{equation*}
for all $i\in [n]$ and $j \in [m]$, which we call the \emph{$\kappa$-bound}.  

\begin{figure}[htp]
\centering
\begin{tikzpicture}
\matrix[square matrix,nodes={draw,
      minimum height=11pt,
      anchor=center,
      text width=11pt,
      align=center,
      inner sep=0pt
    },]{
 0 & 0 & |[fill=black!60]|1 & |[fill=black!60]|1 \\
 0 & 0 & |[fill=black!60]|1 &  |[fill=black!60]|1 \\
 |[fill=black!60]|1 & |[fill=black!60]|1 & |[fill=black!60]|1 &  |[fill=black!60]|1 \\
};
\end{tikzpicture}
\caption{The $3$TDS matrix corresponding to Figure \ref{fi:rook34}.}\label{fi:rook34mat}
\end{figure}

We call a $n \times m$ $(0,1)$-matrix $S$ a \emph{$k$TDS matrix} if it satisfies the $\kappa$-bound for all $i \in [n]$ and $j \in [m]$.  Furthermore, we call $S$ a \emph{min-$k$TDS matrix} if it has exactly $\gamma_{\times k,t}(K_n \Box K_m)$ ones.  Note that a $k$TDS matrix (respectively min-$k$TDS matrix) remains a $k$TDS matrix (respectively min-$k$TDS matrix) under permutations of its rows and/or columns.

\begin{lemma}\label{lm:sparserow}
For $n \geq 1$ and $m \geq 1$, a $n \times m$ $k$TDS matrix with an all-$0$ column or an all-$0$ row has at least $kn$ or $km$ ones, respectively.
\end{lemma}

\begin{proof} Let $S$ be a $n\times m$ $k$TDS matrix. Assume there exists $1\le j_0\le m$ such that $\mathfrak{c}(j_0)=0$. Then to achieve $\kappa(i,j_0) \geq k$ for any $i \in [n]$, we need $\mathfrak{r}(i)\ge k$.  Since this is true for every row in $S$, we must have at least $kn$ ones. A similar argument works if there exists $1\le i_0\le n$ such that $\mathfrak{r}(i_0)=0$.
\end{proof}

There are instances when $kn$ ones is the least number of ones in any $n \times m$ $k$TDS matrix. We establish some cases in the following proposition, see also Theorem 3.3 from \cite{KPS}.

\begin{prop}\label{prop:manycols}
When $m \geq n \geq 2$ and $m \geq k$,
\begin{equation*}\label{eq:Knmsimplebound}
\gamma_{\times k,t}(K_n \Box K_m) \leq kn
\end{equation*}
with equality when $m \geq kn-1$.
%
\end{prop}

\begin{proof}
If $m \geq n \geq 2$ and $m \geq k$, the $n \times m$ $(0,1)$-matrix with ones in the last $k$ columns and zeros elsewhere is a $k$TDS matrix with $kn$ ones.

Assume $m \geq kn-1$ and let $S$ be a $n \times m$ $k$TDS matrix.  If $S$ has a column of zeros, then $|S|\ge kn$ by Lemma~\ref{lm:sparserow}.  If $S$ has no column of zeros but $m\ge kn$, then $|S|\ge kn$.  Thus, assume $m=kn-1$ and $\mathfrak{c}(j)\ge1$ for all $1\le j\le m$.  If $|S|< kn$, then $\mathfrak{c}(j)=1$ for all $1\le j\le m$.  Therefore, if $s_{ij}=1$, then $\mathfrak{r}(i)\ge k+1$ to satisfy $\kappa(i,j) \geq k$.  If this is true for every row, then $|S|\ge(k+1)n > kn$.  Otherwise, there is a row of zeros, and Lemma~\ref{lm:sparserow} implies $|S|\ge km \geq kn$.

\end{proof}



Motivated by \cite{BLL}, given a $(0,1)$-matrix $S$, we can construct a graph $\Gamma(S)$ with vertices corresponding to the ones in $S$ and edges between $2$ ones belonging to the same row or column, if there are no other ones between them.  The following gives one such example.
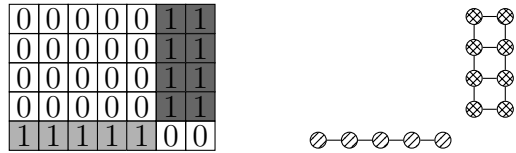
\begin{center}
\begin{tikzpicture}
\matrix[square matrix,nodes={draw,
      minimum height=11pt,
      anchor=center,
      text width=11pt,
      align=center,
      inner sep=0pt
    },]{
 0 &  0 &  0 & 0 & 0 & |[fill=black!60]|1 & |[fill=black!60]| 1 \\
0&  0 &  0 & 0 & 0 & |[fill=black!60]|1 & |[fill=black!60]| 1 \\
0 & 0 & 0 &  0 &  0 &   |[fill=black!60]|1 & |[fill=black!60]| 1 \\
0 & 0 & 0 & 0 &  0 &   |[fill=black!60]|1 & |[fill=black!60]| 1 \\
|[fill=black!30]| 1 &  |[fill=black!30]|1 &  |[fill=black!30]|1 & |[fill=black!30]|1 & |[fill=black!30]|1 & 0 &  0 \\
};
\end{tikzpicture}
\qquad
\raisebox{0.4em}{
\begin{tikzpicture}[scale=0.41]
\tikzstyle{vertex}=[draw,circle,fill=black!60,minimum size=6pt,inner sep=0pt]

\node[vertex,pattern=crosshatch] (v1) at (6,1) {};
\node[vertex,pattern=crosshatch] (v2) at (5,1) {};
\node[vertex,pattern=crosshatch] (v3) at (6,0) {};
\node[vertex,pattern=crosshatch] (v4) at (5,0) {};
\node[vertex,pattern=crosshatch] (v5) at (6,-1) {};
\node[vertex,pattern=crosshatch] (v6) at (5,-1) {};
\node[vertex,pattern=crosshatch] (v7) at (6,-2) {};
\node[vertex,pattern=crosshatch] (v8) at (5,-2) {};

\draw (v1) -- (v2) ;
\draw (v3) -- (v4) ;
\draw (v5) -- (v6) ;
\draw (v7) -- (v8) ;
\draw (v1) -- (v3) -- (v5) -- (v7);
\draw (v2) -- (v4) -- (v6) -- (v8);

\tikzstyle{vertex}=[draw,circle,fill=black!30,minimum size=6pt,inner sep=0pt]

\node[vertex,pattern=north east lines] (u1) at (4,-3) {};
\node[vertex,pattern=north east lines] (u2) at (3,-3) {};
\node[vertex,pattern=north east lines] (u3) at (2,-3) {};
\node[vertex,pattern=north east lines] (u4) at (1,-3) {};
\node[vertex,pattern=north east lines] (u5) at (0,-3) {};

\draw (u1) -- (u2) -- (u3) -- (u4) -- (u5);
\end{tikzpicture}
}
\end{center}
In this way, every $k$TDS matrix $S$ correspond to a graph, which has, in general, several (connected) components.
 If the set of vertices of a component of $\Gamma(S)$ is $\{(i_1,j_1),\dots (i_p,j_p)\}$, then we define the corresponding component of $S$ as the submatrix of $S$ formed by the intersection of rows $\{R_{i_1},\dots,R_{i_p}\}$ and columns $\{C_{j_1},\dots,C_{j_p}\}$, where $R_d$ and $C_d$ are the $d$-th row and column of $S$, respectively.  
 We shade two components in the example above. In this example, the $5 \times 7$ matrix is the union of two components (a $4 \times 2$ component and a $1 \times 5$ component).  A $k$TDS matrix $S$ with a component $H$, up to permutations of the rows and columns of $S$, looks like one of the following
$$
\begin{array}{|c|c|}
\hline
H & \emptyset \\
\hline
\emptyset & ? \\
\hline
\end{array},
\quad
\begin{array}{|c|}
\hline
H \\
\hline
\emptyset \\
\hline
\end{array},
\quad
\begin{array}{|c|c|}
\hline
H & \emptyset \\
\hline
\end{array},
\text{ or}
\quad
\begin{array}{|c|}
\hline
H \\
\hline
\end{array},
$$
where the question mark ($?$) denotes some $(0,1)$-submatrix, and $\emptyset$ denotes an all-$0$ submatrix.  Components of $k$TDS matrices have the following properties
\begin{itemize}
 \item components have no all-$0$ rows and no all-$0$ columns,
 \item components are $k$TDS matrices in their own right.
\end{itemize}

\begin{remark}\label{rem:2onesrowocol} If $S$ is a $3$TDS matrix with no all-$0$ rows and no all-$0$ columns, then in order to achieve the $3$-bound, it has at least $2$ ones in each row or in each column. Moreover,
if $S$ has at least $2$ ones in each row (or column), the same is true for each of its components. Since we are interested in the study of rook's graphs with $m\ge n$, we will assume that each $3$TDS matrix has at least $2$ ones in each row
\end{remark}

In order to describe our main result, we will need the following $3$TDS matrices.
For $x \geq 1$ and $y \geq 1$ such that $x+y\ge5$, we define $J(x,y)$ as the $x \times y$ all-$1$ matrix.

For $x\ge5$, let $D(x,3)$ be the $x\times3$ $3$TDS matrix whose first $x-3$ rows coincide with $(0,1,1)$,
the $(x-2)$-th and $(x-1)$-th rows coincide with $(1,0,1)$, and the last row coincides with $(1,1,0)$.
Depicted below we have the matrix $D(x,3)$ for $x \in \{5, 6, 7\}$
\begin{center}
$\begin{array}{ccc}
\begin{tikzpicture}
\matrix[square matrix]{
|[fill=white]| & |[fill=black!60]| & |[fill=black!60]| \\
|[fill=white]| & |[fill=black!60]| & |[fill=black!60]| \\
|[fill=black!60]| & |[fill=white]| &  |[fill=black!60]| \\
|[fill=black!60]| & |[fill=white]| &  |[fill=black!60]| \\
|[fill=black!60]| & |[fill=black!60]| & |[fill=white]| \\
};
\end{tikzpicture}
&
\begin{tikzpicture}
\matrix[square matrix]{
|[fill=white]| & |[fill=black!60]| & |[fill=black!60]| \\
|[fill=white]| & |[fill=black!60]| & |[fill=black!60]| \\
|[fill=white]| & |[fill=black!60]| & |[fill=black!60]| \\
|[fill=black!60]| & |[fill=white]| &  |[fill=black!60]| \\
|[fill=black!60]| & |[fill=white]| &  |[fill=black!60]| \\
|[fill=black!60]| & |[fill=black!60]| & |[fill=white]| \\
};
\end{tikzpicture}
&
\begin{tikzpicture}
\matrix[square matrix]{
|[fill=white]| & |[fill=black!60]| & |[fill=black!60]| \\
|[fill=white]| & |[fill=black!60]| & |[fill=black!60]| \\
|[fill=white]| & |[fill=black!60]| & |[fill=black!60]| \\
|[fill=white]| & |[fill=black!60]| & |[fill=black!60]| \\
|[fill=black!60]| & |[fill=white]| &  |[fill=black!60]| \\
|[fill=black!60]| & |[fill=white]| &  |[fill=black!60]| \\
|[fill=black!60]| & |[fill=black!60]| & |[fill=white]| \\
};
\end{tikzpicture}
\end{array}$
\end{center}
\section{On the construction of special $3$TDS matrices}

We describe how to construct a special class of $3$TDS matrices, looking with particular attention at the shape of their components. 
Moreover, we compute the number of ones in such matrices.
Notice that these matrices are exactly the ones appearing in Table \ref{ta:smallk3}.

\begin{prop}\label{prop:compclass2}
For any $m\ge n\ge 6$, except $(n,m)=(6,6)$, there exists a $n \times m$ $3$TDS matrix $S$ with no all-$0$ rows and no all-$0$ columns with at least $2$ ones in each row whose components, up to permutations of the rows and columns, are all $J(1,4)$ or $J(3,2)$, except possibly for
\begin{itemize}
\item exactly one $J(4,2)$ component;
\item exactly one $J(1,y)$ component with $5 \leq y \leq 6$;
\item exactly one $D(5,3)$ component but no $J(4,2)$ component;
\item exactly one $D(6,3)$ component and no $D(5,3)$ component or $J(4,2)$ component or $J(1,y)$ component with $5 \leq y \leq 6$.
\end{itemize}
\end{prop}
\begin{proof} Notice that by Table \ref{ta:smallk3}, it is enough to show that
\begin{itemize} 
\item if $S$ is a $n\times m$ $3$TDS matrix with the properties we require, then we have a way to construct $S'$ a $n\times (m+1)$ $3$TDS matrix with the same properties;
\item if $S$ is a $n\times (n+1)$ $3$TDS matrix with the properties we require, then we have a way to construct $S'$ a $(n+1)\times (n+1)$ $3$TDS matrix with the same properties.
\end{itemize}
Let now $S$ be a $n\times m$ $3$TDS matrix with the properties we require. We will apply the following rules to obtain $S'$ a $n\times (m+1)$ $3$TDS matrix with the same properties.
\begin{enumerate}  
\item If $S$ contains a $J(1,6)$ component and a $D(5,3)$ component, we obtain $S'$ by transforming these two components in two $J(1,4)$ components and a $J(4,2)$ component, and leaving the other components unchanged. 
\item If $S$ contains a $J(1,4)$ component and a $D(6,3)$ component, we obtain $S'$ by transforming these two components in one $J(1,4)$ component and two $J(3,2)$ components, and leaving the other components unchanged. 
\item If $S$ contains a $J(1,6)$ component, a $J(4,2)$ component and a $J(3,2)$ component, we obtain $S'$ by transforming these three components in two $J(1,4)$ components and a $D(6,3)$ component, and leaving the other components unchanged. 
\item If $S$ contains a $J(1,6)$ component, two $J(3,2)$ components and no $J(4,2)$ or $D(5,3)$ components, we obtain $S'$ by transforming these three components in two $J(1,4)$ components and a $D(5,3)$ component, and leaving the other components unchanged.
\item In all other cases, since $S$ always contains at least one  $J(1,y)$ component with $4 \leq y \leq 5$, we obtain $S'$ by transforming the $J(1,y)$ component in a $J(1,y+1)$ component, and leaving the other components unchanged. 
\end{enumerate}
Notice that if $S$ has only one $J(1,6)$ component, only one $J(4,2)$ component, or only one $J(3,2)$ component, then $n=4,5$.

Let now $S$ be a $n\times (n+1)$ $3$TDS matrix with the properties we require. We will apply the following rules to obtain $S'$ a $(n+1)\times (n+1)$ $3$TDS matrix with the same properties.
\begin{enumerate}
\item If $S$ contains a $J(1,y)$ component, with $y=5,6$ and a $J(4,2)$ components, we obtain $S'$ by transforming these two components in a $J(1,y-1)$ component and a $D(5,3)$ component, and leaving the other components unchanged. 
\item If $S$ contains a $J(1,y)$ component, with $y=5,6$ and a $D(5,3)$ component, we obtain $S'$ by transforming these two components in a $J(1,y-1)$ component and two $J(3,2)$ components, and leaving the other components unchanged. 
\item If $S$ contains two $J(1,4)$ components and a $D(6,3)$ components, we obtain $S'$ by transforming these three components in a $J(1,6)$ component, a $J(3,2)$ component and a $D(5,3)$ component, and leaving the other components unchanged. 
\item If $S$ contains two $J(1,4)$ components, a $J(4,2)$ component and no $J(1,y)$ components, with $y=5,6$, we obtain $S'$ by transforming these three components in a $J(1,6)$ component and two $J(3,2)$ components, and leaving the other components unchanged. 
\item If $S$ contains a $D(5,3)$ components and no $J(1,y)$ components, with $y=5,6$, we obtain $S'$ by transforming this component in a $D(6,3)$ component, and leaving the other  components unchanged.
\item In all other cases, since $S$ always contains at least one  $J(3,2)$ component, we obtain $S'$ by transforming this component in a $J(4,2)$ component, and leaving the other  components unchanged. 
\end{enumerate}
Notice that if $S$ has only one $J(1,4)$ component, only one $J(4,2)$ component and no $J(1,y)$ components, with $y=5,6$, then $n=5$, or $n=m=8$ or $n> m$.
Similarly, if $S$ has only one $J(1,4)$ component and one $D(6,3)$ component, then $n=m=7$ or $n> m$.
\end{proof}

We can now compute the number of ones in a $3$TDS matrix satisfying the requirements of the previous proposition.

\begin{prop}\label{prop:numberofones}
For any integer $m\ge n \ge 6$, except $(n,m)=(6,6)$, let $2n\equiv  3m+k \pmod {10}$, where $0 \le k \le 9$. Then the number of ones in a matrix of Proposition \ref{prop:compclass2} is given by
$$\begin{cases}
  \lceil \frac{8n+3m}{5} \rceil  & \text{ if $k=0,1,2,3,4,7,8,9;$}  \\
  \lceil \frac{8n+3m}{5} \rceil +1 & \text{ if $k=5,6.$} 
  \end{cases}$$
\end{prop}
\begin{proof}
Let $S$ be a $n\times m$ $3$TDS matrix with no all-$0$ rows and no all-$0$ columns with at least $2$ ones in each row as described in Proposition \ref{prop:compclass2}.
Let $a$ be the number of $J(1,4)$ components in $S$ and let $b$ be the number of $J(3,2)$ components in $S$. To prove our statement we have to analyze six cases.

\textit{Case I}: Assume $S$ has only $J(1,4)$ and $J(3,2)$ components.  Then
\begin{align*}
n &= a+3b, \\
m &= 4a+2b,
\end{align*}
and the number of ones in $S$ is $(4a+6b)=(8n+3m)/5$.  In this case, we have $2n\equiv 3m \pmod {10}$.

\textit{Case II}: Assume $S$ has $J(1,4)$ components, $J(3,2)$ components and one $J(4,2)$ component.  Then
\begin{align*}
n &= a+3b+4, \\
m &= 4a+2b+2,
\end{align*}
and the number of ones in $S$ is  $(4a+6b+8)=(8n+3m+2)/5= \lceil (8n+3m)/5 \rceil$.  In this case, we have $2n\equiv 3m+2 \pmod {10}$.



\textit{Case III}: Assume $S$ has $J(1,4)$ components, $J(3,2)$ components and one $J(1,y)$ component with $5 \leq y \leq 6$. We have
\begin{align*}
n &= a+3b+1, \\
m &= 4a+2b+y,
\end{align*}
and the number of ones in $S$ is
\begin{align*}
(4a+6b)+y &= (8n+3m+2y-8)/5 \\
  &=
  \begin{cases}
 (8n+3m+2)/5= \lceil (8n+3m)/5 \rceil & \text{ if $y=5;$} \\
 (8n+3m+4)/5= \lceil (8n+3m)/5 \rceil & \text{ if $y=6.$} 
  \end{cases}
\end{align*}
 In this case we have $2n \equiv 3m-3y+2 \pmod {10}$, i.e., $2n \equiv 3m+7, 3m+4 \pmod {10}$ when $y=5,6 $, respectively.
 
 \textit{Case IV}: Assume $S$ has $J(1,4)$ components, $J(3,2)$ components, a $J(1,y)$ component with $5 \leq y \leq 6$ and a $J(4,2)$ component. We have
\begin{align*}
n &= a+3b+5, \\
m &= 4a+2b+y+2,
\end{align*}
and the number of ones in $S$ is
\begin{align*}
(4a+6b)+y+8 &= (8n+3m+2y-6)/5 \\
  &=
  \begin{cases}
 (8n+3m+4)/5= \lceil (8n+3m)/5 \rceil & \text{ if $y=5;$} \\
 (8n+3m+6)/5= \lceil (8n+3m)/5 \rceil+1 & \text{ if $y=6.$} 
  \end{cases}
\end{align*}
 In this case we have $2n \equiv 3m-3y+4 \pmod {10}$, i.e., $2n \equiv 3m+9, 3m+6  \pmod {10}$ when $y=5,6 $, respectively.
 
  \textit{Case V}: Assume $S$ has $J(1,4)$ components, $J(3,2)$ components and one $D(x,3)$ component with $5 \leq x \leq 6$. We have
\begin{align*}
n &= a+3b+x, \\
m &= 4a+2b+3,
\end{align*}
and the number of ones in $S$ is
\begin{align*}
(4a+6b)+2x &= (8n+3m+2x-9)/5 \\
  &=
  \begin{cases}
 (8n+3m+1)/5= \lceil (8n+3m)/5 \rceil & \text{ if $x=5;$} \\
 (8n+3m+3)/5= \lceil (8n+3m)/5 \rceil & \text{ if $x=6.$} 
  \end{cases}
\end{align*}
 In this case we have $2n \equiv 3m+2x+1 \pmod {10}$, i.e., $2n \equiv 3m+1, 3m+3 \pmod {10}$ when $x=5,6 $, respectively.
 
 \textit{Case VI}: Assume $S$ has $J(1,4)$ components, $J(3,2)$ components, a $J(1,y)$ component with $5 \leq y \leq 6$ and a $D(5,3)$ component. We have
\begin{align*}
n &= a+3b+6,\\
m &= 4a+2b+y+3,
\end{align*}
and the number of ones in $S$ is
\begin{align*}
(4a+6b)+10+y &= (8n+3m+2y-7)/5 \\
  &=
  \begin{cases}
 (8n+3m+3)/5= \lceil (8n+3m)/5 \rceil & \text{ if $y=5;$} \\
 (8n+3m+5)/5= \lceil (8n+3m)/5 \rceil+1 & \text{ if $y=6.$} 
  \end{cases}
\end{align*}
In this case we have $2n \equiv 3m-3y+3 \pmod {10}$, i.e., $2n \equiv 3m+8, 3m+5 \pmod {10}$ when $y=5,6 $, respectively.
\end{proof}

\begin{remark}\label{rem:n45onesrow} A direct computation shows that when $n\in\{4,5\}$ and $(n,m)=(6,6)$, we can compute the number of ones of the matrices in Table \ref{ta:smallk3} with no all-$0$ rows and no all-$0$ columns with the formula of Proposition \ref{prop:numberofones}.
\end{remark}

\section{Useful inequalities for min-$3$TDS}

We prove several inequalities for $\gamma_{\times 3,t}(K_n \Box K_m)$. Specifically, we show how $\gamma_{\times 3,t}$
changes when, in a $3$TDS matrix, we increase the number of rows or columns in the general case, or both in the square case.
The first lemma describes a lower bound for the number of ones in a $3$TDS matrix. 

\begin{lemma}\label{lemma:nnvs2n+1} Let $m\ge n\ge3$. Then $\gamma_{\times 3,t}(K_n \Box K_m)\ge2n+2$.
\end{lemma}
\begin{proof} Suppose that $\gamma_{\times 3,t}(K_n \Box K_m)\le2n+1$ and let $S$ be a $n\times m$ $3$TDS matrix with $|S|=2n+1$. Since $2n+1<3n$, by Lemma \ref{lm:sparserow}, $S$ has no all-$0$ rows or all-$0$ columns. 
Since by Remark \ref{rem:2onesrowocol} we can assume that $\mathfrak{r}(i)\ge2$ for all $1\le i\le n$, then $S$ has one row with $3$ ones and $n-1$ rows with $2$ ones. 
Without loss of generality, we can assume that the first row of $S$ has $3$ ones in the first three entries. As a consequence,  $\mathfrak{r}(i)=2$ for all $2\le i\le n$.
For all $1\le j\le3$, $\kappa(1,j)=3+\mathfrak{c}(j)-2=\mathfrak{c}(j)+1\ge3$, and hence $\mathfrak{c}(j)\ge 2$, i.e. each of the first three columns of $S$ has at least $2$ ones.
Moreover, if $m\ge4$, since $S$ has no all-$0$ rows or all-$0$ columns, for all $4\le j\le m$ there must exists $2\le i\le n$ such that $S$ has a one in position $(i,j)$. Then $\kappa(i,j)=2+\mathfrak{c}(j)-2=\mathfrak{c}(j)\ge3$,
i.e. each of the last $m-3$ columns of $S$ has at least $3$ ones.

Assume $n=3$. If $m\ge4$, since $\mathfrak{c}(j)\ge 2$ for all $1\le j\le 3$ and $\mathfrak{c}(j)\ge 3$ for all $4\le j\le m$, then $|S|\ge6+3(m-3)=3m-3>2n-1$. We can then assume that $m=3$ and that the zeros of $S$ are in position $(2,2)$ and $(3,1)$. However, $\kappa(2,1)=2$ and so $S$ is not a $3$TDS matrix.

Assume now $n=4$. Since $\mathfrak{c}(4)\ge3$, the last column of $S$ is $(0,1,1,1)^t$. Hence, we can assume that the remaining ones of $S$ are in position $(2,1)$, $(3,2)$ and $(4,3)$. However, $\kappa(2,1)=2$ and so $S$ is not a $3$TDS matrix.

Assume now $n\ge5$. Counting the ones of $S$ by columns we obtain that $|S|\ge 6+3(m-3)=3m-3$. However, since $m\ge n\ge5$, $3m-3>2n+1$ and so $S$ is not a $3$TDS matrix.
\end{proof}

\begin{remark}\label{rem:smallcasesdiag} If $3\le n\le 10$, then by Lemma \ref{lemma:nnvs2n+1}, the $n\times n$ $3$TDS matrix of Table \ref{ta:smallk3} are min-$3$TDS and so  $\gamma_{\times 3,t}(K_n \Box K_n)=2n+2$. 
\end{remark}

We are now able to compute $\gamma_{\times 3,t}(K_3 \Box K_m)$ for all $m\ge3$.
\begin{lemma}\label{lemma:n3case3tds} If $m\ge3$, then 
$$
\gamma_{\times 3,t}(K_3 \Box K_m)=
\begin{cases}
  8 & \text{ if $m=3,4;$} \\
  9 & \text{ if $m\ge 5.$} 
  \end{cases}
$$
\end{lemma}
\begin{proof} By Lemma \ref{lemma:nnvs2n+1}, $\gamma_{\times 3,t}(K_3 \Box K_m)\ge8$. Looking at the $3$TDS matrices in Table \ref{ta:smallk3}, we obtain that $\gamma_{\times 3,t}(K_3 \Box K_3)=\gamma_{\times 3,t}(K_3 \Box K_4)=8$.

Let now $m\ge5$. Suppose that there exists $S$ a $3\times m$ $3$TDS matrix with $|S|=8$. By Remark \ref{rem:2onesrowocol}, this implies that there exists $1\le i\le 3$ such that $\mathfrak{r}(i)=2$. Without loss of generality we can assume that $i=3$
and that the last row of $S$ coincides with $(0,\dots,0,1,1)$. If $m-1\le j\le m$, then $\kappa(3,j)=2+\mathfrak{c}(j)-2\ge3$. This implies that $\mathfrak{c}(m-1)=\mathfrak{c}(m)=3$. Moreover, by Lemma \ref{lm:sparserow}, 
$S$ has no all-$0$ rows or all-$0$ columns, and hence $\mathfrak{c}(j)\ge1$ for all $1\le j\le m-2$. This implies that $|S|=\sum_{j=1}^m\mathfrak{c}(j)\ge (m-2)+6>8$, but this is a contraddiction.
\end{proof}

The following result describes the relation between min-$3$TDS matrices that have the same number of rows but whose number of columns differs by one. 

\begin{lemma}\label{lemma:3dtsgrowsbyrow} Let $m\ge n\ge3$. Then $$\gamma_{\times 3,t}(K_n \Box K_m)\le \gamma_{\times 3,t}(K_n \Box K_{m+1})\le \gamma_{\times 3,t}(K_n \Box K_m)+1.$$
\end{lemma}
\begin{proof} Firstly we will prove the first inequality, i.e. we will prove that $\gamma_{\times 3,t}(K_n \Box K_m)\le \gamma_{\times 3,t}(K_n \Box K_{m+1})$.

If $\gamma_{\times 3,t}(K_n \Box K_{m+1})=3n$, the first inequality holds by Proposition \ref{prop:manycols}. 
Let $S$ be a $n\times (m+1)$ $3$TDS matrix with $|S|=\gamma_{\times 3,t}(K_n \Box K_m)-1<3n$. By Lemma \ref{lm:sparserow}, $S$ has no all-$0$ rows or all-$0$ columns. Furthermore, since $|S|<3n$, then
there exists $1\le j\le m+1$ such that $\mathfrak{c}_S(j)\le2$. Hence, without loss of generality, we can assume that $j=m+1$. This fact is crucial for the rest of the proof.

If $n=3$, the first inequality holds by Lemma \ref{lemma:n3case3tds}. 
Assume now $m\ge n \ge4$. If $\mathfrak{c}_S(m+1)=1$, then we can assume that the last column of $S$ is $(1,0,\dots,0)^t$. Since $\kappa_S(1,m+1)=\mathfrak{r}_S(1)+1-2\ge3$, then $\mathfrak{r}_S(1)\ge 4$. Consider $S'$ the matrix obtained from $S$ by deleting the last column. Notice that $\mathfrak{r}_{S'}(1)\ge 3$. $S'$ is a $n\times m$ matrix with $|S|-1=\gamma_{\times 3,t}(K_n \Box K_m)-2$ ones, and hence $S'$ is not a $3$TDS matrix, by definition of $\gamma_{\times 3,t}$. However, $\kappa_{S'}(i,j)=\kappa_S(i,j)\ge3$, if $2\le i\le n$ and $1\le j\le m$. Since $S'$ is not a $3$TDS matrix, there exists $1\le j\le m$ such that $\kappa_{S'}(1,j)\le2$. This implies that $\mathfrak{r}_{S'}(1)= 3$ and so that $\mathfrak{r}_S(1)= 4$. Since $m+1\ge5$, then the first row of $S$ has at least one zero in the first $m$ entries. We can construct $S''$ a $n\times m$ matrix obtained from $S$ by deleting the last column and putting exactly $1$ one in one of the zeros of the first row. By construction $S''$ is a $n\times m$ $3$TDS matrix with $|S|=\gamma_{\times 3,t}(K_n \Box K_m)-1$ ones, but this is a contradiction. 

If $\mathfrak{c}_S(m+1)=2$, then we can assume that the last column of $S$ is equal to $(1,1,0,\dots,0)^t$. Let $S'$ be the matrix obtained from $S$ by deleting the last column. $S'$ is a $n\times m$ matrix with $|S|-2=\gamma_{\times 3,t}(K_n \Box K_m)-3$ ones, and hence it is not a $3$TDS matrix. However, $\kappa_{S'}(i,j)=\kappa_S(i,j)\ge3$, if $3\le i\le n$ and $1\le j\le m$.
This implies that at least one of the first two rows of $S$ have exactly $3$ ones. Assume it is the first one. Since $m+1\ge5$, then the first  row of $S$ has at least $2$ zeros in the first $m$ entries. 
If any of the first $m$ columns of $S$ have $2$ zeros in the first two rows, we can construct $S''$ a $n\times m$ matrix obtained from $S$ by deleting the last column and putting 2 ones in the first two entries of such column. If such column does not exist, we can construct $S''$ a $n\times m$ matrix obtained from $S$ by deleting the last column and putting exactly $1$ one in one zero of the first row and, if the second row has a zero, $1$ one there. By construction $S''$ is a $n\times m$ $3$TDS matrix with at most $|S|=\gamma_{\times 3,t}(K_n \Box K_m)-1$ ones, but this is a contradiction. This proves the first inequality.

We are now ready to prove the second inequality, i.e. to prove that $\gamma_{\times 3,t}(K_n \Box K_{m+1})\le \gamma_{\times 3,t}(K_n \Box K_m)+1$.
Let $S$ be a minimum $n\times m$ $3$TDS matrix. By Lemma \ref{lemma:nnvs2n+1}, we have that $\gamma_{\times 3,t}(K_n \Box K_m)\ge 2n+2$ and hence there exists $1\le i\le n$ such that $\mathfrak{r}_S(i)\ge3$. Without loss of generality we can assume that $i=n$.
Consider now $S'$ a $n\times (m+1)$ matrix such that the first $m$ columns coincide with $S$ and the last column is $(0,\dots,0,1)^t$. By construction $S'$ is a $n\times (m+1)$ $3$TDS matrix with $|S|+1=\gamma_{\times 3,t}(K_n \Box K_m)+1$ ones.
\end{proof}

The next lemma describes the relation between min-$3$TDS matrix that have the same number of columns but whose number of rows differs by one. 

\begin{lemma}\label{lemma:3dtsgrowsbycolumn} Let $m> n\ge3$, and assume $\gamma_{\times 3,t}(K_n \Box K_m)< 3n$.
Then $$\gamma_{\times 3,t}(K_n \Box K_m)\le \gamma_{\times 3,t}(K_{n+1} \Box K_m)\le \gamma_{\times 3,t}(K_n \Box K_m)+2.$$
\end{lemma}
\begin{proof} 
Firstly we will prove the first inequality, i.e. we will prove that $\gamma_{\times 3,t}(K_n \Box K_m)\le \gamma_{\times 3,t}(K_{n+1} \Box K_m)$.

Let $S$ be a $(n+1)\times m$ $3$TDS matrix with $|S|=\gamma_{\times 3,t}(K_n \Box K_m)-1$. Since $|S|<3n$, by Remark \ref{rem:2onesrowocol} there exists $1\le i\le n+1$ such that $\mathfrak{r}_S(i)=2$. Without loss of generality, we can assume that $i=n+1$ and that the last row of $S$ is $(0,\dots,0,1,1)$. 
Consider $S'$ the matrix obtained from $S$ by deleting the last row. $S'$ is a $n\times m$ matrix with $|S|-2=\gamma_{\times 3,t}(K_n \Box K_m)-3$ ones, and then it is not a $3$TDS matrix. However, $\kappa_{S'}(i,j)=\kappa_S(i,j)\ge3$, if $1\le i\le n$ and $1\le j\le m-2$. This implies that at least one of the last two columns of $S$ have exactly $3$ ones. Assume that this column is the last of $S$. Since $n+1\ge4$, the last column of $S$ has at least one zero in the first $n$ entries. If any of the first $n$ rows of $S$ have $2$ zeros in the last two columns, we can construct $S''$ a $n\times m$ matrix obtained from $S$ by deleting the last row and putting $2$ ones in the last two entries of such row. If such row does not exist but the penultimate column of $S$ has a zero, we can construct $S''$ a $n\times m$ matrix obtained from $S$ by deleting the last row and putting exactly $1$ one in one zero of the penultimate column and exactly $1$ one in one zero of the last column. If the penultimate column has no zero, we can construct $S''$ a $n\times m$ matrix obtained from $S$ by deleting the last row and putting exactly $1$ one in one zero of the last column. By construction $S''$ is a $n\times m$ $3$TDS matrix with at most $|S|=\gamma_{\times 3,t}(K_n \Box K_m)-1$ ones, but this is a contradiction. This proves the first inequality.

We are now ready to prove the second inequality, i.e. to prove that $\gamma_{\times 3,t}(K_{n+1} \Box K_m)\le \gamma_{\times 3,t}(K_n \Box K_m)+2$.
Let $S$ be a minimum $n\times m$ $3$TDS matrix. By assumption, we have that $\gamma_{\times 3,t}(K_n \Box K_m)< 3n$ and hence there exists $1\le i\le n$ such that $\mathfrak{r}_S(i)=2$. Without loss of generality we can assume that $i=n$ and that the last row of $S$ coincides with $(0,\dots,0,1,1)$.
Consider now $S'$ a $(n+1)\times m$ matrix such that the first $n$ rows coincide with $S$ and the last row is $(0,\dots,0,1,1)$. By construction $S'$ is a $(n+1)\times m$ $3$TDS matrix with $|S|+2=\gamma_{\times 3,t}(K_n \Box K_m)+2$ ones.
\end{proof}

We now describes the relation between square min-$3$TDS matrix  whose number of rows and columns both differ by one. 

\begin{lemma}\label{lemma:3dtsgrowsbydiagonal} Let $n\ge3$. Then 
$$\gamma_{\times 3,t}(K_n \Box K_n)+2 \le \gamma_{\times 3,t}(K_{n+1} \Box K_{n+1})\le \gamma_{\times 3,t}(K_n \Box K_n)+3.$$
\end{lemma}
\begin{proof} 
Firstly, we will prove the first inequality, i.e. we will prove that $\gamma_{\times 3,t}(K_n \Box K_n)+2\le \gamma_{\times 3,t}(K_{n+1} \Box K_{n+1})$.

If $n=3,4$, the first inequality follows from Remark \ref{rem:smallcasesdiag}.
Assume $n\ge5$. Suppose there exists $S$ a $(n+1)\times (n+1)$ $3$TDS matrix with $|S|=\gamma_{\times 3,t}(K_n \Box K_n)+1$. By Remark \ref{rem:2onesrowocol} and $\gamma_{\times 3,t}(K_n \Box K_n)+1<3(n+1)$, there exists $1\le i\le n+1$ such that $\mathfrak{r}_S(i)=2$. Without loss of generality we can assume that $i=n+1$ and that the last row of $S$ coincides with $(0,\dots,0,1,1)$. Let now $S'$ be the $n\times (n+1)$ matrix obtained from $S$ by deleting the last row. $S'$ has $|S|-2=\gamma_{\times 3,t}(K_n \Box K_n)-1$ ones, and hence it is not a $3$TDS by Lemma \ref{lemma:3dtsgrowsbyrow}. However, $\kappa_{S'}(i,j)=\kappa_S(i,j)\ge3$ for all $1\le i \le n$ and $1\le j\le n-1$. Since $S$ is a $3$TDS matrix, if $n\le j\le n+1$, then $\kappa_S(n+1,j)=2+\mathfrak{c}_{S}(j)-2=\mathfrak{c}_{S}(j)\ge3$, and hence $\mathfrak{c}_{S'}(j)\ge2$. However, since $S'$ is not a $3$TDS matrix, there exist $1\le i\le n$ and $n\le j\le n+1$ such that $\kappa_{S'}(i,j)=2$, and hence  $\mathfrak{c}_{S'}(n)=2$ or $\mathfrak{c}_{S'}(n+1)=2$. Without loss of generality, we can assume that $(i,j)=(1,n+1)$, and hence that the last column of $S'$ is $(1,1,0,\dots,0)^t$ and $\mathfrak{r}_{S'}(1)=2$. 

Assume that $\mathfrak{c}_{S'}(n)\ge3$. If in $S'$ there is a column with $2$ zeros in the first two entries, we can construct $S''$ a $n\times n$ matrix obtained from $S'$ by deleting the last column and putting $2$ ones in the first two entries of such column. If such column does not exist, then $\mathfrak{r}_{S'}(2)\ge4$. Furthermore, since $|S|<3(n+1)$, there must exists a column with $2$ zeros, one in the first entry and the second one in the $j$-th position, for some $j\ge3$. We can construct $S''$ a $n\times n$ matrix obtained from $S'$ by deleting the last column and putting $1$ one in the first entry and $1$ one in the $j$-th position of such column. By construction $S''$ is a $n\times n$ $3$TDS matrix with $|S'|=\gamma_{\times 3,t}(K_n \Box K_n)-1$ ones, but this is a contradiction.

Assume now that $\mathfrak{c}_{S'}(n)=2$. Denote by $w$ the penultimate column of $S'$. There are four cases.
If $w$ has $2$ zeros in the first two entries, then we can construct $S''$ a $n\times n$ matrix obtained from $S'$ by deleting the last column and putting $2$ ones in the first two entries of $w$. By construction $S''$ is a $n\times n$ $3$TDS matrix with $|S'|=\gamma_{\times 3,t}(K_n \Box K_n)-1$ ones, but this is a contradiction.

If $w=(1,0,\dots)^t$, then the first row of $S'$ is equal to $(0,\dots,0,1,1)$. Since $|S|<3(n+1)$, there must exists $1\le j\le n-1$ such that $\mathfrak{c}_{S'}(j)=\mathfrak{c}_S(j)\ge2$. We can construct $S''$ a $n\times n$ matrix obtained from $S'$ by deleting the last column and putting $1$ one in position $(1,j)$ and $1$ one in position $(2,n)$. By construction $S''$ is a $n\times n$ $3$TDS matrix with $|S'|=\gamma_{\times 3,t}(K_n \Box K_n)-1$ ones, but this is impossible.

Assume $w=(0,1,\dots)^t$. If $\mathfrak{r}_{S'}(2)=2$, then the second row of $S'$ is equal to $(0,\dots,0,1,1)$. Since $|S|<3(n+1)$, there must exists $1\le j\le n-1$ such that $\mathfrak{c}_{S'}(j)\ge2$. We can construct $S''$ a $n\times n$ matrix obtained from $S'$ by deleting the last column and putting $1$ one in position $(1,n)$ and $1$ one in position $(2,j)$. By construction $S''$ is a $n\times n$ $3$TDS matrix with $|S'|=\gamma_{\times 3,t}(K_n \Box K_n)-1$ ones, but this is a contradiction. 
If $\mathfrak{r}_{S'}(2)\ge3$, but there is at least one zero in the second row of $S'$, we can construct $S''$ a $n\times n$ matrix obtained from $S'$ by deleting the last column and putting $1$ one in position $(1,n)$ and exactly $1$ one in one zero of the second row. By construction $S''$ is a $n\times n$ $3$TDS matrix with $|S'|=\gamma_{\times 3,t}(K_n \Box K_n)-1$ ones, but this is a contradiction. 
If $\mathfrak{r}_{S'}(2)=n+1$, we can construct $S''$ a $n\times n$ matrix obtained from $S'$ by deleting the last column and putting $1$ one in position $(1,n)$ and exactly $1$ one in one zero of the first row. By construction $S''$ is a $n\times n$ $3$TDS matrix with $|S'|=\gamma_{\times 3,t}(K_n \Box K_n)-1$ ones, and this is a contradiction.

If $w=(1,1,0,\dots,0)^t$, then the first row of $S'$ is equal to $(0,\dots,0,1,1)$. Since $|S|<3(n+1)$, there must exists $1\le j\le n-1$ such that $\mathfrak{c}_{S'}(j)\ge2$, and we can assume that $S'$ has ones in positions $(p,j)$ and $(q,j)$, for some $2\le p< q \le n$. If $\mathfrak{r}_{S'}(2)=2$, then the second row of $S'$ is equal to $(0,\dots,0,1,1)$. This implies that the first two entries of the $j$-th column of $S'$ are zero. We can construct $S''$ a $n\times n$ matrix obtained from $S'$ by deleting the last column and putting $1$ one in position $(1,j)$, $1$ one in position $(2,j)$, $1$ one in position $(p,n)$ and putting $1$ zero in position $(p,j)$. By construction $S''$ is a $n\times n$ $3$TDS matrix with $|S'|=\gamma_{\times 3,t}(K_n \Box K_n)-1$ ones, however this is a contradiction. 
If $\mathfrak{r}_{S'}(2)=3$, then, without loss of generality, we can assume that the second row of $S'$ is equal to $(0,\dots,0,1,1,1)$. Since $\kappa_{S'}(2,n-1)=\kappa_{S}(2,n-1)\ge3$, this implies that $\mathfrak{c}_{S'}(n-1)\ge2$, and we can assume that $S'$ has ones in positions $(2,n-1)$ and $(i,n-1)$, with $3\le i\le n$. We can construct $S''$ a $n\times n$ matrix obtained from $S'$ by deleting the last column and putting $1$ one in position $(1,n-1)$ and $1$ one in position $(i,n)$. By construction $S''$ is a $n\times n$ $3$TDS matrix with $|S'|=\gamma_{\times 3,t}(K_n \Box K_n)-1$ ones, but this is a contradiction. 
If $\mathfrak{r}_{S'}(2)\ge4$, then, without loss of generality, we can assume that the second row of $S'$ is equal to $(\dots,1,1,1,1)$. We can construct $S''$ a $n\times n$ matrix obtained from $S'$ by deleting the last column and putting $1$ one in position $(1,n-1)$ and $1$ one in position $(1,n-2)$. By construction $S''$ is a $n\times n$ $3$TDS matrix with $|S'|=\gamma_{\times 3,t}(K_n \Box K_n)-1$ ones, but this is a contradiction.

We are now ready to prove the second inequality, i.e. to prove that $\gamma_{\times 3,t}(K_{n+1} \Box K_{n+1})\le \gamma_{\times 3,t}(K_n \Box K_n)+3$.
By Proposition \ref{prop:numberofones}, Remark \ref{rem:n45onesrow} and Lemma \ref{lemma:n3case3tds}, $\gamma_{\times 3,t}(K_{n} \Box K_{n+1})<3n$. By Lemmas \ref{lemma:3dtsgrowsbyrow} and \ref{lemma:3dtsgrowsbycolumn},
$\gamma_{\times 3,t}(K_{n+1} \Box K_{n+1})\le \gamma_{\times 3,t}(K_{n} \Box K_{n+1})+2 \le \gamma_{\times 3,t}(K_n \Box K_n)+3.$
\end{proof}

\begin{remark}\label{rem:onesbycolumnsvsrowk} Let $m\ge n\ge3$ and $S$ a $n\times m$ $3$TDS matrix with no all-$0$ columns or all-$0$ row. Let $1\le k\le n-1$ and assume that $S$ has $2n+k$ ones. Counting by column, this implies that $S$ has at most $k+2\lfloor\frac{k}{2}\rfloor$ columns that contain a one belonging to a row with at least $3$ ones. Hence all the other columns contain at least $1$ one belonging to a row with $2$ ones, and so such columns all have at least $3$ ones. This shows that $|S|\ge(k+2\lfloor\frac{k}{2}\rfloor) +3(m-k-2\lfloor\frac{k}{2}\rfloor)=3m-2k-4\lfloor\frac{k}{2}\rfloor$.
\end{remark}

\section{The square case}

We consider the case when $n=m$ and we give an explicit formula for $\gamma_{\times 3,t}(K_n \Box K_n)$ that is independent from the component structure of square $3$TDS matrices.
Notice that our formula coincides with the number of ones of the square matrices appearing in Table \ref{ta:smallk3}.

\begin{theo}\label{theo:thesquarecase} For any integer $n \ge 3$, let $n \equiv r \pmod {10}$, where $0 \le r \le 9$. Then
$$
\gamma_{\times 3,t}(K_n \Box K_n)=
\begin{cases}
  2n+2\lfloor \frac{n}{10} \rfloor+2  & \text{ if $r=4,5,6,7,8,9;$}  \\
  2n+2\lfloor \frac{n}{10} \rfloor+\lceil \frac{r}{3}\rceil    & \text{ if $r=0,1,2,3$ and $n\ne3;$} \\
  8 & \text{ if $n=3.$}
  \end{cases}
$$
\end{theo}
\begin{proof}If $n=3$, then $\gamma_{\times 3,t}(K_n \Box K_n)=8$ by Lemma \ref{lemma:n3case3tds}. Assume $n\ge4$.
  Since the description of Proposition \ref{prop:numberofones} and Remark \ref{rem:n45onesrow} coincides with our claim when $n=m$, we clearly have that
$$
\gamma_{\times 3,t}(K_n \Box K_n)\le
\begin{cases}
  2n+2\lfloor \frac{n}{10} \rfloor+2  & \text{ if $r=4,5,6,7,8,9;$}  \\
  2n+2\lfloor \frac{n}{10} \rfloor+\lceil \frac{r}{3}\rceil    & \text{ if $r=0,1,2,3$ and $n\ne3.$} 
  \end{cases}
$$

Notice that when $n \equiv r \pmod {10}$ and $r=1,2$, then $(2n+2\lfloor \frac{n}{10} \rfloor+\lceil \frac{r}{3}\rceil)+2=2(n+1)+2\lfloor \frac{n+1}{10} \rfloor+\lceil \frac{r}{3}\rceil$. This implies that if $\gamma_{\times 3,t}(K_n \Box K_n)=2n+2\lfloor \frac{n}{10} \rfloor+\lceil \frac{r}{3}\rceil$, then by Lemma \ref{lemma:3dtsgrowsbydiagonal}, $\gamma_{\times 3,t}(K_{n+1} \Box K_{n+1})=2(n+1)+2\lfloor \frac{n+1}{10} \rfloor+\lceil \frac{r}{3}\rceil$. 
Similarly, when $r=4,5,6,7,8$, then $(2n+2\lfloor \frac{n}{10} \rfloor+2)+2=2(n+1)+2\lfloor \frac{n+1}{10} \rfloor+2$. This implies that if $\gamma_{\times 3,t}(K_n \Box K_n)=2n+2\lfloor \frac{n}{10} \rfloor+2$, then by Lemma \ref{lemma:3dtsgrowsbydiagonal}, $\gamma_{\times 3,t}(K_{n+1} \Box K_{n+1})=2(n+1)+2\lfloor \frac{n+1}{10} \rfloor+2$.  
Moreover, when $r=9$, $(2n+2\lfloor \frac{n}{10} \rfloor+2)+2=2(n+1)+2\lfloor \frac{n+1}{10} \rfloor$. This implies that if $\gamma_{\times 3,t}(K_n \Box K_n)=2n+2\lfloor \frac{n}{10} \rfloor+2$, then by Lemma \ref{lemma:3dtsgrowsbydiagonal}, $\gamma_{\times 3,t}(K_{n+1} \Box K_{n+1})=2(n+1)+2\lfloor \frac{n+1}{10} \rfloor$.
However, when $r=0$, then $(2n+2\lfloor \frac{n}{10} \rfloor)+3=2(n+1)+2\lfloor \frac{n+1}{10} \rfloor+1$, and hence in this situation we need to prove that $\gamma_{\times 3,t}(K_{n+1} \Box K_{n+1})=\gamma_{\times 3,t}(K_n \Box K_n)+3$. Similarly when $r=3$, $(2n+2\lfloor \frac{n}{10} \rfloor+1)+3=2(n+1)+2\lfloor \frac{n+1}{10} \rfloor+2$, and hence also in this situation we need to prove that $\gamma_{\times 3,t}(K_{n+1} \Box K_{n+1})=\gamma_{\times 3,t}(K_n \Box K_n)+3$. By Lemma \ref{lemma:3dtsgrowsbydiagonal}, it is enough to show that if $r=0,3$, then $\gamma_{\times 3,t}(K_{n+1} \Box K_{n+1})>\gamma_{\times 3,t}(K_n \Box K_n)+2$.

Assume $r=0$. Suppose there exists $S$ a $(n+1)\times (n+1)$ $3$TDS matrix with $|S|=\gamma_{\times 3,t}(K_n \Box K_n)+2=2(n+1)+k$, where $k=2\lfloor \frac{n}{10}\rfloor$. By Remark \ref{rem:onesbycolumnsvsrowk}, $|S|\ge 3(n+1)-2k-4\lfloor\frac{k}{2}\rfloor$. Notice that since $r=0$, then $k=\frac{n}{5}$ and it is an even integer. This implies that $|S|\ge3(n+1)-4k$. However, since $k=\frac{n}{5}$, then $2(n+1)+k<3(n+1)-4k$ and hence $S$ is not a $3$TDS matrix. 

Assume now $r=3$. Suppose there exists $S$ a $(n+1)\times (n+1)$ $3$TDS matrix with $|S|=\gamma_{\times 3,t}(K_n \Box K_n)+2=2(n+1)+k$, where $k=2\lfloor \frac{n}{10}\rfloor+1$. By Remark \ref{rem:onesbycolumnsvsrowk}, $|S|\ge 3(n+1)-2k-4\lfloor\frac{k}{2}\rfloor$. Notice that since $r=3$, then $k=\frac{n-3}{5}+1$ and it is an odd integer. This implies that $|S|\ge 3(n+1)-4k+2$. However, since $k=\frac{n-3}{5}+1$, then $2(n+1)+k<3(n+1)-4k+2$ and hence $S$ is not a $3$TDS matrix. 
\end{proof}

\section{The general case}
We give a formula for $\gamma_{\times 3,t}(K_n \Box K_m)$ that coincides with the number of ones of the $3$TDS matrices in Table \ref{ta:smallk3}, but the argument is independent of the shape of the components in a $3$TDS matrix.

\begin{theo}\label{theo:thegeneralcase}
Let $m\ge n \ge 1$. Assume $(n,m)\notin\{(1,1),(1,2),(1,3),$$(2,2)\}$ and $2n\equiv  3m+k \pmod {10}$, where $0 \le k \le 9$. Then
$$
\gamma_{\times 3,t}(K_n \Box K_m)=
\begin{cases}
  4 & \text{$n=1$ and $m\ge4;$} \\
  6 & \text{$n=2$ and $m\ge3;$} \\
  8 & \text{$n=3$ and $m=3,4;$} \\
  3n & \text{if $m\ge \lfloor \frac{7n-1}{3} \rfloor-1;$} \\
  \lceil \frac{8n+3m}{5} \rceil  & \text{ if $k=0,1,2,3,4,7,8,9;$}  \\
  \lceil \frac{8n+3m}{5} \rceil +1 & \text{ if $k=5,6.$} 
  \end{cases}
$$
\end{theo}
\begin{proof} If $(n,m)\in\{(1,1),(1,2),(1,3),(2,2)\}$, then there are no $n\times m$ $3$TDS matrices. If $n=1$ and $m\ge4$, then any $1\times m$ $(0,1)$-matrix with exactly $4$ ones is a min-$3$TDS matrix.
If $n=2$ and $m\ge3$, then any $2\times m$ $(0,1)$-matrix with exactly $3$ columns with $2$ ones is a min-$3$TDS matrix.
If $n=3$ and $m=3, 4$, by Lemma \ref{lemma:n3case3tds}, $\gamma_{\times 3,t}(K_3 \Box K_3)=\gamma_{\times 3,t}(K_3 \Box K_4)=8$.

Assume $n\ge4$. By Propositions \ref{prop:manycols} and \ref{prop:numberofones}, and Remark \ref{rem:n45onesrow}, we have that 
$$
\gamma_{\times 3,t}(K_n \Box K_m)\le
\begin{cases}
  \min\{\lceil \frac{8n+3m}{5} \rceil, 3n\}  & \text{ if $k=0,1,2,3,4,7,8,9;$}  \\
  \min\{\lceil \frac{8n+3m}{5} \rceil +1,3n\} & \text{ if $k=5,6.$} 
  \end{cases}
$$
If $m>\frac{7n-1}{3}-2$ (which occurs when $m\ge \lfloor \frac{7n-1}{3} \rfloor-1$), then $\lceil\frac{8n+3m}{5} \rceil +1\ge 3n$, in which case the previous minimums coincide with $3n$.
Assume now $m< \lfloor \frac{7n-1}{3} \rfloor-1$. Since we assume $m\ge n$, we can write $m=n+d$, for some $d\ge0$. 
When $n=m$, a direct computation shows that our formula coincides with the one of Theorem \ref{theo:thesquarecase}. Hence, using induction on $m$, it is enough to show that if  $\gamma_{\times 3,t}(K_n \Box K_m)$ coincides with our formula, so does $\gamma_{\times 3,t}(K_n \Box K_{m+1})$. We will prove this with a case by case analysis.

\textit{Case I}: Assume $2n\equiv  3m \pmod {10}$ and $m=n+d$. Then we can write $\lceil \frac{8n+3m}{5} \rceil=2n+\lceil\frac{n+3d}{5}\rceil=2n+k$. Notice that $k=\frac{n+3d}{5}$ and it is an even integer. Since $2n\equiv  3m \pmod {10}$, then  $2n\equiv  3(m+1)+7 \pmod {10}$. Moreover, $\lceil \frac{8n+3(m+1)}{5} \rceil=2n+\lceil\frac{n+3d+3}{5}\rceil$ and hence $\lceil \frac{8n+3(m+1)}{5} \rceil=2n+k+1$.
By hypothesis, $\gamma_{\times 3,t}(K_n \Box K_m)=2n+k$. By Lemma \ref{lemma:3dtsgrowsbyrow}, $2n+k\le \gamma_{\times 3,t}(K_n \Box K_{m+1}) \le2n+k+1$. Suppose there exists $S$ a $n\times (m+1)$ $3$TDS matrix with $|S|=2n+k$. By Remark \ref{rem:onesbycolumnsvsrowk}, $|S|\ge3(m+1)-2k-4\lfloor\frac{k}{2}\rfloor$. In this situation, $3(m+1)-2k-4\lfloor\frac{k}{2}\rfloor=3n+3d+3-4(\frac{n+3d}{5})$ and it is strictly bigger than $2n+k$. This implies that $S$ is not a $3$TDS matrix and hence that $\gamma_{\times 3,t}(K_n \Box K_{m+1}) =2n+k+1=\lceil\frac{8n+3(m+1)}{5}\rceil$.

\textit{Case II}: Assume $2n\equiv  3m+7 \pmod {10}$ and $m=n+d$. Then we can write $\lceil \frac{8n+3m}{5} \rceil=2n+\lceil\frac{n+3d}{5}\rceil=2n+k$. Notice that $k=\frac{n+3d+2}{5}$ and it is an odd integer. Since $2n\equiv  3m+7 \pmod {10}$, then  $2n\equiv  3(m+1)+4 \pmod {10}$. Moreover, $\lceil \frac{8n+3(m+1)}{5} \rceil=2n+\lceil\frac{n+3d+3}{5}\rceil$ and hence $\lceil \frac{8n+3(m+1)}{5} \rceil=2n+k+1$.
By hypothesis, $\gamma_{\times 3,t}(K_n \Box K_m)=2n+k$. By Lemma \ref{lemma:3dtsgrowsbyrow}, $2n+k\le \gamma_{\times 3,t}(K_n \Box K_{m+1}) \le2n+k+1$. Suppose there exists $S$ a $n\times (m+1)$ $3$TDS matrix with $|S|=2n+k$. By Remark \ref{rem:onesbycolumnsvsrowk}, $|S|\ge3(m+1)-2k-4\lfloor\frac{k}{2}\rfloor$. In this situation, $3(m+1)-2k-4\lfloor\frac{k}{2}\rfloor=3n+3d+3-4(\frac{n+3d+2}{5})+2$ and it is strictly bigger than $2n+k$. This implies that $S$ is not a $3$TDS matrix and hence that $\gamma_{\times 3,t}(K_n \Box K_{m+1}) =2n+k+1=\lceil\frac{8n+3(m+1)}{5}\rceil$.

\textit{Case III}: Assume $2n\equiv  3m+4 \pmod {10}$ and $m=n+d$. Then we can write $\lceil \frac{8n+3m}{5} \rceil=2n+\lceil\frac{n+3d}{5}\rceil=2n+k$. Notice that $k=\frac{n+3d+4}{5}$ and it is an even integer. Since $2n\equiv  3m+4 \pmod {10}$, then  $2n\equiv  3(m+1)+1 \pmod {10}$. Moreover, $\lceil \frac{8n+3(m+1)}{5} \rceil=2n+\lceil\frac{n+3d+3}{5}\rceil$ and hence $\lceil \frac{8n+3(m+1)}{5} \rceil=2n+k$. By Lemma \ref{lemma:3dtsgrowsbyrow}, $\gamma_{\times 3,t}(K_n \Box K_{m+1}) =\lceil \frac{8n+3(m+1)}{5} \rceil$. 

\textit{Case IV}: Assume $2n\equiv  3m+1 \pmod {10}$ and $m=n+d$. Then we can write $\lceil \frac{8n+3m}{5} \rceil=2n+\lceil\frac{n+3d}{5}\rceil=2n+k$. Notice that $k=\frac{n+3d+1}{5}$ and it is an even integer. Since $2n\equiv  3m+1 \pmod {10}$, then  $2n\equiv  3(m+1)+8 \pmod {10}$. Moreover, $\lceil \frac{8n+3(m+1)}{5} \rceil=2n+\lceil\frac{n+3d+3}{5}\rceil$ and hence $\lceil \frac{8n+3(m+1)}{5} \rceil=2n+k+1$.
By hypothesis, $\gamma_{\times 3,t}(K_n \Box K_m)=2n+k$. By Lemma \ref{lemma:3dtsgrowsbyrow}, $2n+k\le \gamma_{\times 3,t}(K_n \Box K_{m+1}) \le2n+k+1$. Suppose there exists $S$ a $n\times (m+1)$ $3$TDS matrix with $|S|=2n+k$. By Remark \ref{rem:onesbycolumnsvsrowk}, $|S|\ge3(m+1)-2k-4\lfloor\frac{k}{2}\rfloor$. In this situation, $3(m+1)-2k-4\lfloor\frac{k}{2}\rfloor=3n+3d+3-4(\frac{n+3d+1}{5})$ and it is strictly bigger than $2n+k$. This implies that $S$ is not a $3$TDS matrix and hence that $\gamma_{\times 3,t}(K_n \Box K_{m+1}) =2n+k+1=\lceil\frac{8n+3(m+1)}{5}\rceil$.

\textit{Case V}: Assume $2n\equiv  3m+8 \pmod {10}$ and $m=n+d$. Then we can write $\lceil \frac{8n+3m}{5} \rceil=2n+\lceil\frac{n+3d}{5}\rceil=2n+k$. Notice that $k=\frac{n+3d+3}{5}$ and it is an odd integer. Since $2n\equiv  3m+8 \pmod {10}$, then  $2n\equiv  3(m+1)+5 \pmod {10}$. Moreover, $\lceil \frac{8n+3(m+1)}{5} \rceil+1=2n+\lceil\frac{n+3d+3}{5}\rceil+1$ and hence $\lceil \frac{8n+3(m+1)}{5} \rceil=2n+k+1$.
By hypothesis, $\gamma_{\times 3,t}(K_n \Box K_m)=2n+k$. By Lemma \ref{lemma:3dtsgrowsbyrow}, $2n+k\le \gamma_{\times 3,t}(K_n \Box K_{m+1}) \le2n+k+1$. Suppose there exists $S$ a $n\times (m+1)$ $3$TDS matrix with $|S|=2n+k$. By Remark \ref{rem:onesbycolumnsvsrowk}, $|S|\ge3(m+1)-2k-4\lfloor\frac{k}{2}\rfloor$. In this situation, $3(m+1)-2k-4\lfloor\frac{k}{2}\rfloor=3n+3d+3-4(\frac{n+3d+3}{5})+2$ and it is strictly bigger than $2n+k$. This implies that $S$ is not a $3$TDS matrix and hence that $\gamma_{\times 3,t}(K_n \Box K_{m+1}) =2n+k+1=\lceil\frac{8n+3(m+1)}{5}\rceil$.

\textit{Case VI}: Assume $2n\equiv  3m+5 \pmod {10}$ and $m=n+d$. Then we can write $\lceil \frac{8n+3m}{5} \rceil+1=2n+\lceil\frac{n+3d}{5}\rceil+1=2n+k$. Notice that $k=\frac{n+3d}{5}$ and it is an odd integer. Since $2n\equiv  3m+5 \pmod {10}$, then  $2n\equiv  3(m+1)+2 \pmod {10}$. Moreover, $\lceil \frac{8n+3(m+1)}{5} \rceil=2n+\lceil\frac{n+3d+3}{5}\rceil$ and hence $\lceil \frac{8n+3(m+1)}{5} \rceil=2n+k$. By Lemma \ref{lemma:3dtsgrowsbyrow}, $\gamma_{\times 3,t}(K_n \Box K_{m+1}) =\lceil \frac{8n+3(m+1)}{5} \rceil$. 

\textit{Case VII}: Assume $2n\equiv  3m+2 \pmod {10}$ and $m=n+d$. Then we can write $\lceil \frac{8n+3m}{5} \rceil=2n+\lceil\frac{n+3d}{5}\rceil=2n+k$. Notice that $k=\frac{n+3d+2}{5}$ and it is an even integer. Since $2n\equiv  3m+2 \pmod {10}$, then  $2n\equiv  3(m+1)+9 \pmod {10}$. Moreover, $\lceil \frac{8n+3(m+1)}{5} \rceil=2n+\lceil\frac{n+3d+3}{5}\rceil$ and hence $\lceil \frac{8n+3(m+1)}{5} \rceil=2n+k+1$.
By hypothesis, $\gamma_{\times 3,t}(K_n \Box K_m)=2n+k$. By Lemma \ref{lemma:3dtsgrowsbyrow}, $2n+k\le \gamma_{\times 3,t}(K_n \Box K_{m+1}) \le2n+k+1$. Suppose there exists $S$ a $n\times (m+1)$ $3$TDS matrix with $|S|=2n+k$. By Remark \ref{rem:onesbycolumnsvsrowk}, $|S|\ge3(m+1)-2k-4\lfloor\frac{k}{2}\rfloor$. In this situation, $3(m+1)-2k-4\lfloor\frac{k}{2}\rfloor=3n+3d+3-4(\frac{n+3d+2}{5})$ and it is strictly bigger than $2n+k$. This implies that $S$ is not a $3$TDS matrix and hence that $\gamma_{\times 3,t}(K_n \Box K_{m+1}) =2n+k+1=\lceil\frac{8n+3(m+1)}{5}\rceil$.

\textit{Case VIII}: Assume $2n\equiv  3m+9 \pmod {10}$ and $m=n+d$. Then we can write $\lceil \frac{8n+3m}{5} \rceil=2n+\lceil\frac{n+3d}{5}\rceil=2n+k$. Notice that $k=\frac{n+3d+4}{5}$ and it is an odd integer. Since $2n\equiv  3m+9 \pmod {10}$, then  $2n\equiv  3(m+1)+6 \pmod {10}$. Moreover, $\lceil \frac{8n+3(m+1)}{5} \rceil+1=2n+\lceil\frac{n+3d+3}{5}\rceil+1$ and hence $\lceil \frac{8n+3(m+1)}{5} \rceil=2n+k+1$.
By hypothesis, $\gamma_{\times 3,t}(K_n \Box K_m)=2n+k$. By Lemma \ref{lemma:3dtsgrowsbyrow}, $2n+k\le \gamma_{\times 3,t}(K_n \Box K_{m+1}) \le2n+k+1$. Suppose there exists $S$ a $n\times (m+1)$ $3$TDS matrix with $|S|=2n+k$. By Remark \ref{rem:onesbycolumnsvsrowk}, $|S|\ge3(m+1)-2k-4\lfloor\frac{k}{2}\rfloor$. In this situation, $3(m+1)-2k-4\lfloor\frac{k}{2}\rfloor=3n+3d+3-4(\frac{n+3d+4}{5})+2$ and it is strictly bigger than $2n+k$. This implies that $S$ is not a $3$TDS matrix and hence that $\gamma_{\times 3,t}(K_n \Box K_{m+1}) =2n+k+1=\lceil\frac{8n+3(m+1)}{5}\rceil$.

\textit{Case IX}: Assume $2n\equiv  3m+6 \pmod {10}$ and $m=n+d$. Then we can write $\lceil \frac{8n+3m}{5} \rceil+1=2n+\lceil\frac{n+3d}{5}\rceil+1=2n+k$. Notice that $k=\frac{n+3d+1}{5}$ and it is an odd integer. Since $2n\equiv  3m+6 \pmod {10}$, then  $2n\equiv  3(m+1)+3 \pmod {10}$. Moreover, $\lceil \frac{8n+3(m+1)}{5} \rceil=2n+\lceil\frac{n+3d+3}{5}\rceil$ and hence $\lceil \frac{8n+3(m+1)}{5} \rceil=2n+k$. By Lemma \ref{lemma:3dtsgrowsbyrow}, $\gamma_{\times 3,t}(K_n \Box K_{m+1}) =\lceil \frac{8n+3(m+1)}{5} \rceil$. 

\textit{Case X}: Assume $2n\equiv  3m+3 \pmod {10}$ and $m=n+d$. Then we can write $\lceil \frac{8n+3m}{5} \rceil=2n+\lceil\frac{n+3d}{5}\rceil=2n+k$. Notice that $k=\frac{n+3d+3}{5}$ and it is an even integer. Since $2n\equiv  3m+3 \pmod {10}$, then  $2n\equiv  3(m+1) \pmod {10}$. Moreover, $\lceil \frac{8n+3(m+1)}{5} \rceil=2n+\lceil\frac{n+3d+3}{5}\rceil$ and hence $\lceil \frac{8n+3(m+1)}{5} \rceil=2n+k$. By Lemma \ref{lemma:3dtsgrowsbyrow}, $\gamma_{\times 3,t}(K_n \Box K_{m+1}) =\lceil \frac{8n+3(m+1)}{5} \rceil$. 
\end{proof}

Directly from the formula of Theorem \ref{theo:thegeneralcase}, we can generalize the statement of Lemma \ref{lemma:3dtsgrowsbydiagonal} and obtain the following.

\begin{corol} Let $m\ge n\ge3$. Then 
$$\gamma_{\times 3,t}(K_n \Box K_m)+2 \le \gamma_{\times 3,t}(K_{n+1} \Box K_{m+1})\le \gamma_{\times 3,t}(K_n \Box K_m)+3.$$
\end{corol}

\begin{remark} Since, in general, $3n-1>\lfloor \frac{7n-1}{3} \rfloor-1$, we obtain a better bound than the one described in Proposition \ref{prop:manycols} for the case when $\gamma_{\times 3,t}(K_n \Box K_m)=3n$.
\end{remark}

\begin{landscape}
\begin{table}[htp]
\centering
$
\begin{array}{|c|c|c|c|c|c|c|c|c|c|c|c|}
\cline{2-11}
\multicolumn{1}{c|}{} & m=3 & m=4 & m=5 & m=6 & m=7 & m=8 & m=9 & m=10 & m=11 & m=12 \\
\hline
\end{array}$
\caption{Small min-$3$TDS matrices.}\label{ta:smallk3}
\end{table}
\end{landscape}


\begin{thebibliography}{20}

\bibitem{bondy2008graph} J.~A. Bondy and U.~S.~R. Murty. {\em Graph theory},  {\em Graduate texts in mathematics}. Vol. 244, Springer Science and Media, 2008.

\bibitem{BDG} B. Bre\v{s}ar, P. Dorbec, W. Goddard, B. Hartnell, M. Henning, S. Klav\v{z}ar, D. Rall.  Vizing's conjecture: A survey and recent results, {\em J. Graph Theory.} {\bf 69} No.1 (2012)

\bibitem{BM} B. Bre\v{s}ar, M. A. Henning, and S. Klav\v{z}ar. On integer domination in graphs and Vizing-like problems, {\em Taiwanese J. Math.} {\bf 10} No.5 (2006) 1317-1328.

\bibitem{BHRo} B. Bre\v{s}ar, M. A. Henning, D. F. Rall, Paired-domination of Cartesian products of graphs, {\em Util. Math.} {\bf 73} (2007) 255-265.

\bibitem{B2011} P. A. Burchett, On the border queens problem and $k$-tuple domination on the rook's graph, {\em  Congr. Numer.}, {\bf 209} (2011): 179-187.

\bibitem{BLL} P. A. Burchett, D. Lane, J. A. Lachniet, $k$-tuple and $k$-domination on the rook's graph and other results, {\em  Congr. Numer.}, {\bf 199} (2009): 187-204.

\bibitem{CM} K. Choudhary, S. Margulies, I.V. Hicks, A note on total and paired domination of Cartesian product graphs, {\em Electron. J. Combin.} {\bf 20} No.3 (2013).

\bibitem{CMH} K. Choudhary, S. Margulies, I. V. Hicks, Integer domination of Cartesian product graphs, {\em Discrete Math.} {\bf 338} No.7 (2015).

\bibitem{CS1}  E. W. Clark, S. Suen, An inequality related to Vizing's Conjecture, {\em Electron. J. Combin.} {\bf 7} Note.4 (2000).

\bibitem{FRDM} D. C. Fisher, J. Ryan, G. Domke, and A. Majumdar. Fractional domination of strong direct products. {\em Discrete Appl. Math.} {\bf 50} No.1 (1994) 89-91.

\bibitem{HH3}  F. Harary, T. W. Haynes, Double domination in graphs, {\em Ars Combin.} {\bf 55} (2000) 201-213.

\bibitem{HHS5}  T. W. Haynes, S. T. Hedetniemi, P. J. Slater, {\em Fundamentals of Domination in Graphs}, Marcel Dekker, New York, 1998.

\bibitem{HHS6}  T. W. Haynes, S. T. Hedetniemi, P. J. Slater, {\em Domination in Graphs: Advanced Topics}, Marcel Dekker, New York, 1998.

\bibitem{HK7}  M. A. Henning, A. P. Kazemi, $k$-tuple total domination in cross products of graphs, {\em J. Comb. Optim.} {\bf 24} (2011) 339-346.

\bibitem{HK8}  M. A. Henning, A. P. Kazemi, $k$-tuple total domination in graphs, {\em Discrete Appl. Math.} {\bf 158} (2010) 1006-1011.

\bibitem{HR9}  M. A. Henning, D. F. Rall, On the total domination number of Cartesian products of graphs, {\em Graphs Combin.} {\bf 21} (2005) 63-69.

\bibitem{HPT} P. T. Ho, A note on the total domination number. {\em Util. Math.} {\bf 77} (2008).

\bibitem{HJ} X. M. Hou, F. Jiang, Paired domination of Cartesian products of graphs, {\em J. Math. Res. Exposition} {\bf 30} No.1 (2010) 181-185.

\bibitem{HL} X. M. Hou, Y. Lu, On the $\{k\}$-domination number of Cartesian products of graphs, {\em Discrete Math.} {\bf 309} (2009) 3413-3419.

\bibitem{IK} W. Imrich, S. Klav\v{z}ar, {\em Product Graphs: Structure and Recognition}, John Wiley \& Sons, New York, 2000.

\bibitem{KP}  A. P. Kazemi, B. Pahlavsay, $k$-tuple total domination in supergeneralized Petersen graphs, {\em Comm. Math. App.} {\bf 2} No.1 (2011) 21-30.

\bibitem{KPS}  A. P. Kazemi, B. Pahlavsay, R. J. Stones, Cartesian product graphs and $k$-tuple total domination, {\em Filomat}.

\bibitem{Viz} V. G. Vizing, Some unsolved problems in graph theory, {\em Usp. Mat. Nauk} {\bf 23} (1968) 117-134.
\end{thebibliography}
\end{document}